\documentclass[12pt]{amsart}

\usepackage{amsthm,amsfonts,amsmath,amssymb,latexsym,epsfig,mathrsfs,yfonts,marvosym}
\usepackage[usenames]{color}
\usepackage{graphicx}
\usepackage[all]{xy}
\usepackage{epsfig}
\usepackage{epic}
\usepackage{eepic}
\usepackage{hyperref}
\usepackage{dsfont}\let\mathbb\mathds
\usepackage[english]{babel}

\DeclareMathAlphabet\oldmathcal{OMS}        {cmsy}{b}{n}
\SetMathAlphabet    \oldmathcal{normal}{OMS}{cmsy}{m}{n}
\DeclareMathAlphabet\oldmathbcal{OMS}       {cmsy}{b}{n}
 
\usepackage{eucal}

\usepackage[active]{srcltx}

\newtheorem{theorem}{Theorem}[section]
\newtheorem{lemma}[theorem]{Lemma}
\newtheorem{proposition}[theorem]{Proposition}
\newtheorem{corollary}[theorem]{Corollary}

\newtheorem{def/prop}[theorem]{Definition/Proposition}

\newenvironment{example}{\medskip \refstepcounter{theorem}
\noindent  {\bf Example \thetheorem}.\rm}{\,}
\newenvironment{remark}{\medskip \refstepcounter{theorem}
\newcommand     {\comment}[1]   {}
\newcommand{\mute}[2] {}
\newcommand     {\printname}[1] {}

\noindent  {\bf Remark \thetheorem}.\rm}{\,}
\newtheorem*{ack}{Acknowledgements}

\def\<{\langle}

\def\>{\rangle}

\def\fract#1#2{\raise4pt\hbox{$ #1 \atop #2 $}}
\def\decdnar#1{\phantom{\hbox{$\scriptstyle{#1}$}}
\left\downarrow\vbox{\vskip15pt\hbox{$\scriptstyle{#1}$}}\right.}

\def\bbc{{\mathbb C}}

\def\bbp{{\mathbb P}}
\def\bbq{{\mathbb Q}}
\def\bbr{{\mathbb R}}

\def\bbz{{\mathbb Z}}

\def\gra{\alpha}
\def\grb{\beta}

\def\grg{\gamma}

\def\gro{\omega}

\def\grs{\sigma}

\def\grD{\Delta}

\def\grL{\Lambda}

\def\bfp{{\bf p}}

\def\bfv{{\bf v}}
\def\bfw{{\bf w}}

\def\cald{{\mathcal D}}

\def\cali{{\mathcal I}}

\def\calp{{\mathcal P}}

\def\cals{{\oldmathcal S}}

\def\la#1{\hbox to #1pc{\leftarrowfill}}
\def\ra#1{\hbox to #1pc{\rightarrowfill}}

\def\gt{{\mathfrak t}}

\def\gI{{\mathfrak I}}

\def\hook{\mathbin{\hbox to 6pt{%
                 \vrule height0.4pt width5pt depth0pt
                 \kern-.4pt
                 \vrule height6pt width0.4pt depth0pt\hss}}}

\def\12{\xi_{k_1,k_2}}
\def\m5{M^5_{k_1,k_2}}

\begin{document}

\title{On the Topology of some Sasaki-Einstein Manifolds}

\author{Charles P. Boyer and Christina W. T{\o}nnesen-Friedman}\thanks{Both authors were partially supported by grants from the Simons Foundation, CPB by (\#245002) and CWT-F by (\#208799)}
\address{Charles P. Boyer, Department of Mathematics and Statistics,
University of New Mexico, Albuquerque, NM 87131.}
\email{cboyer@math.unm.edu} 
\address{Christina W. T{\o}nnesen-Friedman, Department of Mathematics, Union
College, Schenectady, New York 12308, USA } \email{tonnesec@union.edu}



\maketitle

\markboth{Topology of Sasaki-Einstein Manifolds}{Charles P. Boyer and Christina W. T{\o}nnesen-Friedman}


\begin{abstract}
This is a sequel to our paper \cite{BoTo14a} in which we concentrate on developing some of the topological properties of Sasaki-Einstein manifolds. In particular, we explicitly compute the cohomology rings for several cases not treated in \cite{BoTo14a} and give a formula for homotopy equivalence in one particular 7-dimensional case. 
\end{abstract}


\section{Introduction}

Recently the authors have been able to obtain many new results on extremal Sasakian geometry \cite{BoTo12b,BoTo11,BoTo13,BoTo14a} by giving a geometric construction that combines the `join construction' of \cite{BG00a,BGO06} with the `admissible construction of Hamiltonian 2-forms' for extremal K\"ahler metrics described in \cite{ApCaGa06,ACGT04,ACGT08,ACGT08c}. The current paper is a result of re-arranging the two previous ArXiv papers \cite{BoTo13bArX,BoTo14aArx}. The basic analysis of both the constant scalar curvature and Sasaki-Einstein cases were combined in \cite{BoTo14a} which also contains the foundational topological description. The current paper contains further results on the topology of the Sasaki-Einstein manifolds most of which appeared in \cite{BoTo13bArX}, but were left out of \cite{BoTo14a}.

The main result concerning Sasaki-Einstein manifolds in \cite{BoTo14a} is:

\begin{theorem}\label{admjoinse}
Let $M_{l_1,l_2,\bfw}=M\star_{l_1,l_2}S^3_\bfw$ be the $S^3_\bfw$-join with a regular Sasaki manifold $M$ which is an $S^1$-bundle over a compact positive K\"ahler-Einstein manifold $N$ with a primitive K\"ahler class $[\gro_N]\in H^2(N,\bbz)$. Assume that the relatively prime positive integers $(l_1,l_2)$ are the relative Fano indices given explicitly by 
$$l_{1}(\bfw)=\frac{\cali_N}{\gcd(w_1+w_2,\cali_N)},\qquad   l_2(\bfw)=\frac{w_1+w_2}{\gcd(w_1+w_2,\cali_N)},$$ where $\cali_N$ denotes the Fano index of $N$. Then for each vector $\bfw=(w_1,w_2)\in \bbz^+\times\bbz^+$ with relatively prime components satisfying $w_1>w_2$ there exists a Reeb vector field $\xi_\bfv$ in the 2-dimensional $\bfw$-Sasaki cone on $M_{l_1,l_2,\bfw}$ such that the corresponding Sasakian structure $\cals=(\xi_\bfv,\eta_\bfv,\Phi,g)$ is Sasaki-Einstein (SE).  
\end{theorem}

The procedure involved taking a join of a regular Sasaki-Einstein manifold $M$ with the weighted 3-sphere $S^3_\bfw$, that is, $S^3$ with its standard contact structure, but with a weighted contact 1-form whose Reeb vector field generates rotations with generally different weights $w_1,w_2$ for the two complex coordinates $z_1,z_2$ of $S^3\subset \bbc^2$. We call this the $S^3_\bfw$-join. By the $\bfw$-Sasaki cone we mean the two dimensional subcone of Sasaki cone induced by the Sasaki cone of $S^3_\bfw$. It is denoted by $\gt_\bfw^+$ and can be identified with the open first quadrant in $\bbr^2$.

The SE metrics obtained from Theorem \ref{admjoinse} were obtained earlier by physicists \cite{GHP03,GMSW04a,GMSW04b,CLPP05,MaSp05b} working on the AdS/CFT correspondence. Their method, particularly that of \cite{GMSW04b}, is very closely related to the Hamiltonian 2-form approach of \cite{ApCaGa06} (cf. Section 4.3 of \cite{Spa10}). In fact Thorem \ref{admjoinse} indicates that the physicist's results fit naturally into our geometric construction. Furthermore, we showed in \cite{BoTo14a} that our geometric approach leads naturally to an algorithm for computing the cohomology ring of the $2n+3$-manifolds. In the present paper we explicitly compute the cohomology ring of all such examples of SE manifolds in dimension $7$ showing that there are a countably infinite number of distinct homotopy types of such manifolds.

Most of the SE structures in Theorem \ref{admjoinse} are irregular. Such structures have irreducible transverse holonomy \cite{HeSu12b}, implying there can be no generalization of the join procedure to the irregular case. We must deform within the Sasaki cone to obtain them. Furthermore, it follows from \cite{RoTh11,CoSz12} that constant scalar curvature Sasaki metrics (hence, SE) imply a certain K-semistability.

Of particular interest is the join $M^{2r+3}_{l_1,l_2,\bfw}=S^{2r+1}\star_{l_1,l_2}S^3_\bfw$ of the standard odd dimensional sphere with the weighted $S^3_\bfw$ where 
\begin{equation}\label{sphjoinls}
(l_1,l_2)=\bigl(\frac{r+1}{\gcd(w_1+w_2,r+1)},\frac{w_1+w_2}{\gcd(w_1+w_2,r+1)}\bigr).
\end{equation}
By Theorem 4.5 of \cite{BoTo14a} its cohomology ring is
\begin{equation}\label{sphjoincohr}
\bbz[x,y]/(w_1w_2l_1(\bfw)^2x^2,x^{r+1},x^2y,y^2)
\end{equation}
where $x,y$ are classes of degree $2$ and $2r+1$, respectively.

For the manifolds $M^{7}_{\bfw}$ of dimension 7 ($r=2$) much more is known about the topology. These are special cases of what are called generalized Witten spaces in \cite{Esc05}. In particular, the homotopy type was given in \cite{Kru97}, while the homeomorphism and diffeomorphism type was given in \cite{Esc05}. For our subclass admitting Sasaki-Einstein metrics we give necessary and sufficient conditions on $\bfw$ for homotopy equivalence when the order of $H^4$ is odd in Proposition \ref{homequivprop} below. Thus, we answer in the affirmative the existence of Einstein metrics on certain generalized Witten manifolds.

Aside from the $\bfw$-join with the 5-sphere, for dimension 7 our method produces Sasaki-Einstein manifolds $M_{k,\bfw}^7$ on lens space bundles over all del Pezzo surfaces $\bbc\bbp^2\# k\overline{\bbc\bbp}^2$ that admit a K\"ahler-Einstein metric. In particular 

\begin{theorem}\label{delPezzo}
For each relatively prime pair $(w_1,w_2)$ there exist Sasaki-Einstein metrics on the 7-manifolds $M^7_{k,\bfw}$ with cohomology ring
$$H^q(M^7_{k,\bfw},\bbz)\approx \begin{cases}
                    \bbz & \text{if $q=0,7$;} \\
                    \bbz^{k+1} & \text{if $q=2,5$;} \\
                    \bbz^k_{w_1+w_2}\times \bbz_{w_1w_2} & \text{if $q=4;$} \\
                     0 & \text{if otherwise}, 
                     \end{cases}$$
with the ring relations determined by $\gra_i\cup \gra_j=0, w_1w_2s^2=0,\break (w_1+w_2)\gra_i\cup
s=0,$ and $\gra_i,s$ are the $k+1$ two classes with $i=1,\cdots k$ where $k=3,\cdots,8$. Furthermore, when $4\leq k\leq 8$ the local moduli space of Sasaki-Einstein metrics has real dimension $4(k-4)$.
\end{theorem}

\begin{ack}
The authors would like to thank David Calderbank for discussions and Matthias Kreck for providing us with a proof that the first Pontrjagin class is a homeomorphism invariant.
\end{ack}

\section{The $\bfw$-Sasaki cone when $c_1(\cald)=0$}
In this section we describe some of the properties of the $\bfw$-Sasaki cone when $c_1(\cald)=0$, ending with some examples.
Since we require that $N$ be a positive K\"ahler-Einstein manifold, we have $c_1(N)=\cali_N[\gro_N]$ where $\cali_N$ is the Fano index. Recall from \cite{BoTo14a} that the cohomological Einstein condition $c_1(\cald_{l_1,l_2,\bfw})=0$ implies 
\begin{lemma}\label{c10}
Necessary conditions for the Sasaki manifold $M_{l_1,l_2,\bfw}$ to admit a Sasaki-Einstein metric is that $\cali_N>0$, and that  
$$l_2=\frac{|\bfw|}{\gcd(|\bfw|,\cali_N)},\qquad l_{1}=\frac{\cali_N}{\gcd(|\bfw|,\cali_N)}.$$
\end{lemma}
The integers $l_1,l_2$ in Lemma \ref{c10} were called {\it relative Fano indices} in \cite{BG00a}. For the remainder of the paper we assume that these integers take the values given by Lemma \ref{c10} unless explicitly stated otherwise.

A natural question that arises is whether the $\bfw$-cone contains a regular Reeb vector field.
\begin{proposition}\label{numreg}
Assume $\bfw\neq (1,1)$ and let $K=\gcd(\cali_N,|\bfw|)$. Then there are exactly $K-1$ different $\bfw$-Sasaki cones that have a regular Reeb vector field. These are given by
\begin{equation}\label{regw}
\bfw =\Bigl(\frac{K+n}{\gcd(K+n,K-n)},\frac{K-n}{\gcd(K+n,K-n)}\Bigr).
\end{equation}
where $1\leq n<K$.
\end{proposition}

\begin{proof}
By Proposition 3.4 of \cite{BoTo14a} a $\bfw$-Sasaki cone contains a regular Reeb vector field if and only if there is $n\in \bbz^+$ such that 
$$w_1-w_2=n\frac{w_1+w_2}{\gcd(\cali_N,w_1+w_2)}.$$
Clearly, for a solution we must have $n<\gcd(\cali_N,w_1+w_2)$. Then we have a solution if and only if 
$$(K-n)w_1=(K+n)w_2$$
for all $1\leq n<K$. Since $w_1>w_2$ and they are relatively prime we have the unique solution Equation \eqref{regw} for each integer $1\leq n<K$.
\end{proof}

We have an immediate corollary to Proposition \ref{numreg}:
\begin{corollary}\label{cali1cor}
If $\cali_N=1$ there are no regular Reeb vector fields in any $\bfw$-Sasaki cone with $\bfw\neq (1,1)$.
\end{corollary}

\begin{example}\label{Findex}
Let us determine the $\bfw$-joins with regular Reeb vector field for $\cali_N=2,3$.
For example, if $\cali_N=2$ for a solution to Equation \eqref{regw} we must have $K=2$ which gives
$n=1$ and $\bfw=(3,1)$. This has as a consequence Corollary \ref{Ypqcor} below. 
Similarly if $\cali_N=3$ we must have $K=3$, which gives two solutions $\bfw=(2,1)$ and $\bfw=(5,1)$.
\end{example}

\begin{example}\label{Ypq}
Recall the contact structures $Y^{p,q}$ on $S^2\times S^3$ first studied in the context of Sasaki-Einstein metrics in \cite{GMSW04a}, where $p$ and $q$ are relatively prime positive integers satisfying $p>1$ and $q<p$.
Since in this case the manifold $M_{l_1,l_2,\bfw}=M^3\star_{l_1,l_2}S^3_\bfw$ is $S^2\times S^3$, we have $N=S^2$ with its standard (Fubini-Study) K\"ahler structure. Hence, $\cali_N=2$. 

This was treated in Example 4.7 of \cite{BoPa10} although the conventions\footnote{In particular, there we chose $w_1\leq w_2$; whereas, here we use the opposite convention, $w_1\geq w_2$.} are slightly different. Here we set
\begin{equation}\label{pqw}
\bfw=\frac{1}{\gcd(p+q,p-q)}\bigl(p+q,p-q\bigr).
\end{equation}
Note that the conditions on $p,q$ eliminate the case $\bfw=(1,1)$. We claim that the following relations hold:
\begin{equation}\label{pqrel}
l_1=\gcd(p+q,p-q),\qquad l_2=p.
\end{equation}
To see this we first notice that Lemma \ref{c10} implies that to have a Sasaki-Einstein metric we must have $l_1=1$ and $l_2=\frac{|\bfw|}{2}$ if $|\bfw|$ is even, and $l_1=2$ and $l_2=|\bfw|$ if $|\bfw|$ is odd. Thus, the second of Equations (\ref{pqrel}) follows from the first and Equation (\ref{pqw}). To prove the first equation we first notice that since $p$ and $q$ are relatively prime, $\gcd(p+q,p-q)$ is either $1$ or $2$. Next from Equation (\ref{pqw}) we note that $2p=\gcd(p+q,p-q)|\bfw|$. So if $|\bfw|$ is odd, then $\gcd(p+q,p-q)$ must be even, hence $2$. So the first equation holds in this case. This also shows that if $|\bfw|$ is odd then $p=|\bfw|$ is also odd. Now if $|\bfw|$ is even then both $p+q$ and $p-q$ must be odd. But  then $\gcd(p+q,p-q)$ must be $1$. In this case $p=\frac{|\bfw|}{2}$ which can be either odd or even.

Then from Example \ref{Findex} above we must have $\bfw=(3,1)$ which from Equation \eqref{pqw} means $(p,q)=(2,1)$, so we recover the following result of \cite{BoPa10}:
\begin{corollary}\label{Ypqcor}
For $Y^{p,q}$ the $\bfw$-Sasaki cone has a regular Reeb vector field if and only if $p=2,q=1$.
\end{corollary}
We remark that the  quotient of $Y^{2,1}$ by the regular Reeb vector field is $\bbc\bbp^2$ blown-up at a point; whereas, we have arrived at it from the $\bfw$-Sasaki cone of an $S^1$ orbibundle over $S^2\times\bbc\bbp^1[3,1]$.

For the general $Y^{p,q}$ the Reeb vector field of Theorem 4.2 of \cite{BoPa10} is equivalent to choosing $\bfv=(1,1)$ here. However, as stated in Corollary \ref{Ypqcor}, it is regular only when $p=2,q=1$. Otherwise, it is quasi-regular with ramification index $m=m_1=m_2=p$ if $p$ is odd, and $m=\frac{p}{2}$ if $p$ is even. This is what is called almost regular in \cite{BoTo14a}. In this case we have
 $$s=\gcd(l_2,w_1-w_2)=\gcd(p,\frac{2q}{l_1}).$$
So $s=1$ if $l_1=2$ that is when $|\bfw|$ is odd which also implies that $p$ is odd. Whereas, if $|\bfw|$ is even, $l_1=1$, so $s=\gcd(p,2q)$ which is $2$ if $p$ is even and $1$ if $p$ is odd. Now consider $n$. We have 
$$n=\frac{l_1}{s}(w_1-w_2)=\frac{2q}{s}.$$
Thus, $n=2q$ when $p$ is odd, and $n=q$ when $p$ is even in which case $q$ must be odd. Moreover, from Equation (\ref{ramind}) we have
$$m=\frac{|\bfw|}{s\gcd(|\bfw|,\cali_N)}= 
\begin{cases} p~ \text{if $p$ is odd} \\
                \frac{p}{2} ~\text{if $p$ is even.}
                \end{cases}$$
So in this case our log pair is $(S_{2q},\grD)$ with 
$$\grD=\bigl(1-\frac{1}{p}\bigr)(D_1+D_2)$$
if $p$ is odd, and $(S_q,\grD)$ with $q$ odd and
$$\grD=\bigl(1-\frac{2}{p}\bigr)(D_1+D_2)$$
if $p$ is even. Here $S_{2q}$ and $S_q$ are the even and odd Hirzebruch surfaces, respectively. Note for $Y^{2,1}$ there is no branch divisor, so it is regular as we know from Corollary \ref{Ypqcor} and $S_1$ is $\bbc\bbp^2$ blown-up at a point.
\end{example}

\section{The Topology of the Sasaki-Einstein manifolds}
We briefly recall the method used in \cite{BoTo14a} to prove Theorem \ref{admjoinse}. The idea is that if we know the differentials in the spectral sequence of the fibration 
\begin{equation}\label{MNspec}
M\ra{1.5}N\ra{1.5}\mathsf{B}S^1,
\end{equation}
we can use the commutative diagram of fibrations
\begin{equation}\label{orbifibrationexactseq}
\begin{matrix}M\times S^3_\bfw &\ra{2.6} &M_{l_1,l_2,\bfw}&\ra{2.6}
&\mathsf{B}S^1 \\
\decdnar{=}&&\decdnar{}&&\decdnar{\psi}\\
M\times S^3_\bfw&\ra{2.6} & N\times\mathsf{B}\bbc\bbp^1[\bfw]&\ra{2.6}
&\mathsf{B}S^1\times \mathsf{B}S^1\, 
\end{matrix} \qquad \qquad
\end{equation}
to compute the cohomology ring of the join $M_{l_1,l_2,\bfw}$. Here $\mathsf{B}G$ is the classifying space of a group $G$ or Haefliger's classifying space \cite{Hae84} of an orbifold if $G$ is an orbifold.

\subsection{Examples in General Dimension}
The example when $M$ is a regular Sasakian sphere $S^{2r+1}$ was worked out in \cite{BoTo14a} and further studied in \cite{BoTo14b}. Here we treat the 7-dimensional case in more detail. However, before doing so we give the following result for $M^{2r+3}_\bfw=S^{2r+1}\star_{l_1,l_2}S^3_\bfw$ with $(l_1(\bfw),l_2(\bfw))$ satisfying the conditions of Theorem \ref{admjoinse}.

\begin{lemma}\label{H4lem}
If $H^4(M^{2r+3}_\bfw,\bbz)=H^4(M^{2r+3}_{\bfw'},\bbz)$, then $w'_1w'_2=w_1w_2$ and $l_1(\bfw')=l_1(\bfw)$.
\end{lemma}

\begin{proof}
The equality of the 4th cohomology groups together with the definition of $l_1$ imply
$$w'_1w'_2l_1\gcd(|\bfw|,r+1)^2=w_1w_2\gcd(|\bfw'|,r+1)^2.$$
Set $g_\bfw=\gcd(|\bfw|,r+1)$ and $g_{\bfw'}=\gcd(|\bfw'|,r+1)$. Assume $g_{\bfw'}>1$. Since $\gcd(w'_1,w'_2)=1$, $g_{\bfw'}$ does not divide $w'_1w'_2$. Thus, $g_{\bfw'}^2$ divides $g_\bfw^2$. Interchanging the roles of $\bfw'$ and $\bfw$ gives $g_{\bfw'}=g_\bfw$ which implies $l_1(\bfw')=l_1(\bfw)$, and hence, the lemma in the case that $g_{\bfw'}>1$. Now assume $g_{\bfw'}=1$. Then we have $w_1w_2=w'_1w'_2g_\bfw^2$ which implies that $g_\bfw$ divides $w_1w_2$. But then since $w_1,w_2$ are relatively prime, we must have $g_\bfw=1$.
\end{proof}

Let us set $W=w_1w_2$, and write the prime decomposition of $W=w_1w_2=p_1^{a_1}\cdots p_k^{a_k}$ Let $P_k$ be the number of partitions  of $W$ into the product $w_1w_2$ of unordered relatively prime integers, including the pair $(w_1w_2,1)$. Then a counting argument gives $P_k=2^{k-1}$. Once counted we then order the pair $w_1>w_2$ as before. Let $\calp_W$ denote the set of $(2r+3)$-manifolds $M^{2r+3}_\bfw$ with isomorphic cohomology rings. Then Lemma \ref{H4lem} implies that the cardenality of $\calp_W$ is $P_k=2^{k-1}$. This proves 
\begin{proposition}\label{cor1}
Let $k$ denote the length of the prime decomposition of $w_1w_2$, then there are $2^{k-1}$ simply connected Sasaki-Einstein manifolds $M^{2r+3}_\bfw=S^{2r+1}\star_{l_1,l_2}S^3_\bfw$ of dimension $2r+3$ with isomorphic cohomology rings such that $H^4$ has order $w_1w_2l_1(\bfw)^2$.
\end{proposition}

Generally we can construct the join of any Sasaki-Einstein manifold with the standard $S^3$ to obtain new Sasaki-Einstein manifolds as done in \cite{BG00a} in a rather simple way; the base is a product K\"ahler-Einstein manifold. In the present paper we take the join with a weighted $S^3_\bfw$ and then deform in the Sasaki-cone. It was shown in \cite{BoTo14a} that there is a unique Sasaki-Einstein metric in the $\bfw$-Sasaki cone. From the topological viewpoint taking the join with the weighted $S^3_\bfw$ usually adds torsion. However, in the simplest example which occurs in dimension $5$ the differentials in the spectral sequence conspire to cancel the occurrence of torsion. Of course, in this dimension the only positive K\"ahler-Einstein $2$-manifold is $N=\bbc\bbp^1$ with its standard Fubini-Study K\"ahler-Einstein structure. Then our procedure gives the $5$-manifolds $Y^{p,q}$ discovered by the physicists \cite{GMSW04a} which are diffeomorphic to $S^2\times S^3$ for all relatively prime positive integers $p,q$ such that $1<q<p$. This case has been well studied from various perspectives \cite{Boy11,BoPa10,MaSp05b,CLPP05}. The $7$-dimensional case is quite amenable to further study, and we shall concentrate our efforts in this direction. However, before doing so we shall give some partial results for some examples in which the regular Sasaki manifold $M$ is complicated by the appearance of torsion. 

Note that if we replace the standard odd dimensional sphere by a  rational homology sphere $V^{2r+1}$ with a regular Sasakian structure the computations in \cite{BoTo14a} immediately give 

\begin{proposition}\label{ratcohring}
The rational cohomology ring of the $S^3_\bfw$-join $V^{2r+1}\star_{l_1,l_2,\bfw}S^3_\bfw$ of a rational homology sphere $V^{2r+1}$ is 
$$\bbq[x,y]/(x^2,y^2)$$
where $x,y$ are classes of degree $2$ and $2r+1$, respectively. Here the $l_1,l_2$ are any positive integers satisfying $\gcd(l_2,w_1w_2l_1)=1$.
\end{proposition}

Examples of rational homology spheres with regular Sasaki-Einstein metrics are given in \cite{BGN02a}. They are the Sasakian homogeneous Stiefel manifolds $V_2(\bbr^{2n+1})$ of 2-frames in $\bbr^{2n+1}$ and the 3-Sasakian homogeneous 11-manifold\footnote{The reason for the subscript $+$ on $Sp(1)$ is that there are two non-conjugate $Sp(1)$ subgroups in the exceptional Lie group $G_2$ which we denote by the subscripts $\pm$. The $Sp(1)_-$ is equivalent to $V_2(\bbr^7)$.} $G_2/Sp(1)_+$. Since we want the join to have a Sasaki-Einstein metric somewhere it the Sasaki cone, we require that the pair $(l_1,l_2)$ to be the relative Fano indices of Lemma \ref{c10}. 

\begin{example}\label{Stiefel}
The Stiefel manifold $V_2(\bbr^{2n+1})$ of dimension $4n-1$. It is a circle bundle over the odd complex quadric $Q_{2n-1}(\bbc)$. Its Fano index $\cali$ is $2n-1$. So the relative Fano indices are 
\begin{equation}\label{VrelF}
l_1(\bfw)=\frac{2n-1}{\gcd(|\bfw|,2n-1)},\qquad l_2(\bfw)=\frac{|\bfw|}{\gcd(|\bfw|,2n-1)}.
\end{equation}
Moreover, the cohomology of $V_2(\bbr^{2n+1})$ is 
$$H^p(V_2(\bbr^{2n+1}),\bbz)\approx \begin{cases}
                                                             \bbz &\text{if $p=0,4n-1$;} \\
                                                          \bbz_2&\text{if $p=2n$;} \\ 
                                                                  0 &\text{otherwise.}
                                                                  \end{cases}$$ 
From the long exact homotopy sequence and the commutative diagram one easily obtains the partial results for the join $M_{l_1,l_2,\bfw}(V)=V_2(\bbr^{2n+1})\star_{l_1,l_2}S^3_\bfw$ when $n>2$, namely $M_{l_1,l_2,\bfw}(V)$ is simply connected, $H^3(M_{l_1,l_2,\bfw}(V),\bbz)=H^5(M_{l_1,l_2,\bfw}(V),\bbz)=0$, and 
\begin{equation}\label{Stjoincoh}
\pi_1(M_{l_1,l_2,\bfw}(V))=H^2(M_{l_1,l_2,\bfw}(V),\bbz)\approx \bbz, \quad H^4(M_{l_1,l_2,\bfw}(V),\bbz)\approx \bbz_{w_1w_2l_1^2}.
\end{equation} 
Since the Stiefel manifolds  $V_2(\bbr^{2n+1})$ are $S^1$-bundles over a complex quadric, they are special cases of the next example.                                                            
\end{example}

\begin{example}\label{Ferex} The Sasakian circle bundle $\cals_{d,n+1}$ over a Fermat hypersurface \cite{BG00a}. The  Fermat hypersurface $F_{d,n+1}$ of degree $d$ in $\bbc\bbp^{n+1}$ is described by the hypersurface
\begin{equation}\label{Fhyp}
z_0^d+z_1^d+\cdots +z_{n+1}^d=0.
\end{equation}
It is Fano when $d\leq n+1$ with index $\cali_{F_{d,n+1}}=n+2-d$ when $d\leq n+1$. Moreover, they have a K\"ahler-Einstein metric when $\frac{n+1}{2}\leq d\leq n+1$. Note that $F_{2,2n}$ is the complex quadric $Q_{2n-1}\subset \bbc\bbp^{2n}$ and $\cals_{2,2n}=V_2(\bbr^{2n+1})$ described in Example \ref{Stiefel}. The integral cohomology ring of $F_{d,n+1}$ is well understood \cite{KuWo80}. It is torsion free with $H_*(F_{d,n+1},\bbz)=H_*(\bbp^n,\bbz)$ except in the middle dimension $n$ where the $nth$ cohomology group of $F_{d,n+1}$ is $\bbz^{b_n}$ when $n$ is odd, and $\bbz^{b_n+1}$ when $n$ is even, where
$$b_n=(-1)^n\Bigl(1+{(1-d)^{n+2}-1\over d}\Bigr).$$
Then if $n>4$ we see as in Example \ref{Stiefel} that the join $M_{l_1,l_2,\bfw}=\cals_{d,n+1}\star_{l_1,l_2}S^3_\bfw$ is simply connected satisfying the conditions of Equation \eqref{Stjoincoh}. In order that the join has a Sasaki-Einstein metric in its $\bfw$-Sasaki cone, we must choose the relative Fano indices to be
\begin{equation}\label{FrelF}
l_1(\bfw)=\frac{n+2-d}{\gcd(|\bfw|,n+2-d)},\qquad l_2(\bfw)=\frac{|\bfw|}{\gcd(|\bfw|,n+2-d)}
\end{equation}
with $\frac{n+1}{2}\leq d\leq n+1$ or $d=2$. For $2<d<\frac{n+1}{2}$ it is unknown whether there is such an SE metric.

\end{example}

\begin{example}\label{G2ex} The homogeneous $3$-Sasakian 11-manifold $G_2/Sp(1)_+$. By Proposition 2.3 in \cite{BG00a} the Fano index $\cali$ associated to $G_2/Sp(1)_+$ is $3$.

\end{example}

\subsection{Examples in Dimension $7$}
We focus attention to dimension seven in which case $N$ is a del Pezzo Surface, namely  $\bbc\bbp^2,\bbc\bbp^1\times \bbc\bbp^1$, and $\bbc\bbp^2$ blown-up at $k$ generic points with $1\leq k\leq 8$. Then the $S^3_\bfw$-join a Sasakian circle bundle over $N$ will be a Sasaki 7-manifold.

\subsubsection{$M=S^5, N=\bbc\bbp^2$}\label{Ncp2}
This is a special case of Theorem \ref{setop}. For $\bbc\bbp^2$ with its standard Fubini-Study K\"ahlerian structure, we have $\cali_N=3$. From Example \ref{Findex} we see that we have a regular Reeb vector field in the $\bfw$-Sasaki cone in precisely two cases, either $\bfw=(2,1)$, or $\bfw=(5,1)$. In the first case the relative Fano indices are $(l_1,l_2)=(1,1)$ while in the second case they are $(l_1,l_2)=(1,2)$. In the former case our 7-manifold $M^7_{(2,1)}=S^5\star_{1,1}S^3_{(2,1)}$ is an $S^3$-bundle over $\bbc\bbp^2$; whereas, in the latter case the 7-manifold $M^7_{(5,1)}=S^5\star_{1,2}S^3_{(5,1)}$ is an $L(2;5,1)$ bundle over $\bbc\bbp^2$. Moreover, it follows from standard lens space theory that $L(2;5,1)$ is diffeomorphic to the real projective space $\bbr\bbp^3$. For general $\bfw$ we have

\begin{proposition}\label{topcp2}
Let $M^7_{\bfw}$ be a simply connected 7-manifold of Theorem \ref{setop}. There are two cases:
\begin{enumerate}
\item $3$ divides $|\bfw|$ which implies $l_2=\frac{|\bfw|}{3}$ and $l_1=1$.
\item $3$ does not divide $|\bfw|$ in which case $l_2=|\bfw|$ and $l_1=3$.
\end{enumerate}
In both cases the cohomology ring is given by
$$\bbz[x,y]/(w_1w_2l_1^2x^2,x^3,x^2y,y^2)$$
where $x,y$ are classes of degree $2$ and $5$, respectively.
\end{proposition}

Notice that in case (1) of Proposition \ref{topcp2} $H^4(M^7_\bfw,\bbz)=\bbz_{w_1w_2}$, whereas in case (2) we have $H^4(M^7_{\bfw},\bbz)=\bbz_{9w_1w_2}$. Since $3$ must divide $w_1+w_2$ in the first case and $w_1w_2$ are relatively prime, their cohomology rings are never isomorphic.

\begin{remark}
Let us make a brief remark about the homogeneous case $\bfw=(1,1)$ with symmetry group $SU(3)\times SU(2)\times U(1)$. There is a unique solution with a Sasaki-Einstein metric as shown in \cite{BG00a}. However, dropping both the Einstein and Sasakian conditions, Kreck and Stolz \cite{KS88} gave a diffeomorphism and homeomorphism classification. Furthermore, using the results of \cite{WaZi90}, they show that in certain cases each of the 28 diffeomorphism types admits an Einstein metric. If we drop the Einstein condition and allow contact bundles with non-trivial $c_1$ we can apply the classification results of \cite{KS88} to the Sasakian case. This will be studied elsewhere. 
\end{remark}

For dimension 7 we see from Proposition \ref{topcp2} that if $3$ divides $w_1+w_2$ then the order $|H^4|$ is $W$. However, if $3$ does not divide $w_1+w_2$ then the order of $|H^4|$ is $9W$. So by Lemma \ref{H4lem} $\calp_W$ splits into two cases, $\calp_W^0$ if $W+1$ is divisible by $3$, and $\calp_W^1$ if $W+1$ is not divisible by $3$. Of course, in either case the cardenality of $\calp_W$ is $2^{k-1}$ where $k$ is the number of prime powers in the prime decomposition of $W$.

\begin{proposition}\label{homequivprop}
Suppose the order of $H^4$ is odd. The elements $M^7_\bfw$ and $M^7_{\bfw'}$ in $\calp_W^0$ are homotopy inequivalent if and only if either 
\begin{equation*}
\bigl(\frac{w'_1+w'_2}{3}\bigr)^3\equiv \pm\bigl(\frac{w_1+w_2}{3}\bigr)^3 \mod \bbz_{W}.
\end{equation*}
The elements $M^7_\bfw$ and $M^7_{\bfw'}$ in $\calp_W^1$ are homotopy inequivalent if and only if  
\begin{equation*}
(w'_1+w'_2)^3 \equiv\pm (w_1+w_2)^3 \mod \bbz_{9W}.
\end{equation*}
\end{proposition}      

\begin{proof}
For $r=2$ consider the $E_6$ differential $d_6(\grb)=l_2(\bfw)^3s^3$ in the spectral sequence of Thorem 4.5 of  \cite{BoTo14a}. Since $l_2$ is relatively prime to $l_1(\bfw)^2w_1w_2$, this takes values in the multiplicative group $\bbz_{l_1^2W}^*$ of units in $\bbz_{l_1^2W}$. Taking into account the choice of generators, it takes its values in $\bbz_{l_1^2W}^*/\{\pm 1\}$. According to Theorem 5.1 of \cite{Kru97} $M^7_\bfw,M^7_{\bfw'}\in\calp_W$ are homotopy equivalent if and only if $l_2(\bfw')^3= l_2(\bfw)^3$ in $\bbz^*_{l_1^2W}/\{\pm 1\}$. Of course, this means that $l_2(\bfw')^3=\pm l_2(\bfw)^3$ in $\bbz^*_{l_1^2W}$. Note that the the other two conditions of Theorem 5.1 of \cite{Kru97} are automatically satisfied in our case.
\end{proof}

Using a Maple program we have checked some examples for homotopy equivalence which appears to be quite sparse. So far we haven't found any examples of a homotopy equivalence. However, we have not done a systematic computer search which we leave for future work.

\begin{example}\label{ex1} Our first example is an infinite sequence of pairs with the same cohomology ring. Set $W=3p$ with $p$ an odd prime not equal to $3$, which gives $P_k=2$. Then for each odd prime $p\neq 3$ there are  two manifolds in $\calp_W^1$, namely $M^7_{(3p,1)}$ and $M^7_{(p,3)}$. The order of $H^4$ is $27p$. We check the conditions of Proposition \ref{homequivprop}. We find
$$(3p+1)^3\equiv 9p+1 \mod 27p, \qquad (p+3)^3\equiv p^3+9p^2+27 \mod 27p.$$
First we look for integer solutions of $p^3+9p^2-9p+26\equiv 0 \mod 27p.$ By the rational root test the solutions could only be $p=2,13,26$ none of which are solutions. Next we check the second condition of Proposition \ref{homequivprop}, namely, 
$p^3+9p^2+9p+28\equiv 0 \mod 27p.$
Again by the rational root test we find the only possibilities are $p=2,7,14,28$, from which s we see that there are no solutions. Thus, we see that $M^7_{(3p,1)}$ and $M^7_{(p,3)}$ are not homotopy equivalent for any odd $p\neq 3$.

By the same arguments one can also show that the infinite sequence of pairs of the form $M^7_{(9p,1)}$ and $M^7_{(p,9)}$, with $p$ an odd prime relatively prime to $3$, are never homotopy equivalent.
\end{example}

\begin{remark}\label{pairrem}
In Example \ref{ex1} we do not need to have $p$ a prime, but we do need it to be relatively prime to $3$. In this more general case, there will be more elements in $\calp_W^1$. For example, if $p=55$ we have $P_k=4$ and the pair $(M^7_{(165,1)},M^7_{(55,3)}$ has the same cohomology ring as $M^7_{(33,5)}$ and $M^7_{(15,11)}$. However, they are not homotopy equivalent to either member of the pair nor to each other.
\end{remark}

\begin{example}\label{ex2} A somewhat more involved example is obtained by setting $W=5\cdot 7\cdot 11\cdot 17$. Here $P_k=8$, so this gives eight 7-manifolds in $\calp_W^0$, namely, 
$$M^7_{(6545,1)},M^7_{(1309,5)},M^7_{(935,7)},M^7_{(595,11)},M^7_{(385,17)},M^7_{(187,35)},M^7_{(119,55)},,M^7_{(85,77)}. $$
One can check that these do not satisfy the conditions for homotopy equivalence of Proposition \ref{homequivprop}. So they are all homotopy inequivalent.
\end{example}

It is easy to get a necessary condition for homeomorphism.

\begin{proposition}
Suppose $w_1'w_2'=w_1w_2$ is odd and that $M^7_\bfw$ and $M^7_{\bfw'}$ are homeomorphic. Then in addition to the conditions of Proposition \ref{homequivprop}, we must have 
$$2(w'_1+w'_2)^2\equiv 2(w_1+w_2)^2 \mod 3w_1w_2.$$
\end{proposition}

\begin{proof}
This is because the first Pontrjagin class $p_1$ is actually a homeomorphism invariant\footnote{This appears to be a folklore result with no proof anywhere in the literature. It is stated without proof on page 2828 of \cite{Kru97} and on page 31 of \cite{KrLu05}. We thank Matthias Kreck for providing us with a proof that $p_1$ is a homeomorphism invariant.}. From Kruggel \cite{Kru97} we see that if $3$ does not divide $|\bfw|$
\begin{equation}\label{p1}
p_1(M^7_\bfw)\equiv 3|\bfw|^2-9w_1^2-9w_2^2\equiv -6|\bfw|^2 \mod 9w_1w_2,
\end{equation}
which implies the result in this case. If $3$ divides $|\bfw|$ we have 
\begin{equation}\label{p1'}
p_1(M^7_\bfw)\equiv -6\Bigl(\frac{|\bfw|}{3}\Bigr)^2 \mod w_1w_2
\end{equation}
and this implies the same result.
\end{proof}

Note that Equations \eqref{p1} and \eqref{p1'} both imply the third condition of Theorem 5.1 in \cite{Kru97} holds in our case. To determine a full homeomorphism and diffeomorphism classification requires the Kreck-Stolz invariants \cite{KS88} $s_1,s_2,s_3\in \bbq/\bbz$. These can be determined as functions of $\bfw$ in our case by using the formulae in \cite{Esc05,Kru05}; however, they are quite complicated and the classification requires computer programing which we leave for future work.

It is interesting to compare the Sasaki-Einstein 7-manifolds described by Theorem \ref{setop} with the 3-Sasakian 7-manifolds studied in \cite{BGM94,BG99} for their cohomology rings have the same form. Seven dimensional manifolds whose cohomology rings are of this type were called 7-manifolds of type $r$ in \cite{Kru97} where $r$ is the order of $H^4$.  First recall that the 3-Sasakian 7-manifolds in \cite{BGM94} are given by a triple of pairwise relatively prime positive integers $(p_1,p_2,p_3)$ and $H^4$ is isomorphic to $\bbz_{\grs_2(\bfp)}$ where $\grs_2(\bfp)=p_1p_2+p_1p_3+p_2p_3$ is the second elementary symmetric function of $\bfp=(p_1,p_2,p_3)$. It follows that $\grs_2$ is odd. 
The following theorem is implicit in \cite{Kru97}, but we give its simple proof here for completeness.

\begin{theorem}\label{pwnothomequiv}
The 7-manifolds $M^7_\bfp$ and $M^7_\bfw$ are not homotopy equivalent for any admissible $\bfp$ or $\bfw$.
\end{theorem}

\begin{proof}
These manifolds are distinguished by $\pi_4$. Our manifolds $M^7_\bfw$ are quotients of $S^5\times S^3$ by a free $S^1$-action, whereas, the manifolds $M^7_\bfp$ of \cite{BGM94} are free $S^1$ quotients of $SU(3)$. So from their long exact homotopy sequences we have $\pi_i(M^7_\bfw)\approx \pi_i(S^5\times S^3)$ and $\pi_i(M^7_\bfp)\approx \pi_i(SU(3))$ for all $i>2$. But it is known that $\pi_4(SU(3))\approx 0$ whereas, $\pi_4(S^5\times S^3)\approx \bbz_2$.
\end{proof}

\subsubsection{$M=S^2\times S^3, N=\bbc\bbp^1\times \bbc\bbp^1$} Note that this is Example \ref{Ferex} with $n=d=2$.
We have $\cali_N=2$, so there are two cases: $|\bfw|$ is odd impying $l_2=|\bfw|$ and $l_1=2$; and $|\bfw|$ is even with $l_2=\frac{|\bfw|}{2}$ and $l_1=1$. In both cases the smoothness condition $\gcd(l_2,l_1w_i)=1$ is satisfied. The $E_2$ term of the Leray-Serre spectral sequence of the top fibration of diagram (\ref{orbifibrationexactseq}) is 
$$E^{p,q}_2=H^p(\mathsf{B}S^1,H^q(S^2\times S^3\times S^3_\bfw,\bbz))\approx \bbz[s]\otimes\bbz[\gra]/(\gra^2)\otimes \grL[\grb,\grg],$$ 
which by the Leray-Serre Theorem converges to $H^{p+q}(M_{l_1,l_2,\bfw},\bbz)$. Here $\gra$ is a 2-class and $\grb,\grg$ are 3-classes. From the bottom fibration in Diagram (\ref{orbifibrationexactseq}) we have $d_2(\grb)=\gra\otimes s_1$ and $d_4(\grg)=w_1w_2s^2_2$. From the commutativity of diagram \eqref{orbifibrationexactseq} we have $d_2(\grb)=l_2s$ and $d_4(\grg_\bfw)=w_1w_2l_1^2s^2$ which gives $E_4^{4,0}\approx \bbz_{w_1w_2l_1^2}$, $E_4^{0,3}\approx \bbz$,  $E_4^{2,2}\approx \bbz_{l_2}$,  and $E_\infty^{0,3}=0$. Then using Poincar\'e duality and universal coefficients we obtain

\begin{proposition}\label{MNprop}
In this case $M^7_{l_1,l_2,\bfw}$ with either $(l_1,l_2)=(2,|\bfw|)$ or $(1,\frac{|\bfw|}{2})$ has the cohomology ring given by
$$H^*(M^7_{l_1,l_2,\bfw},\bbz)=\bbz[x,y,u,z]/(x^2,l_2xy,w_1w_2l_1^2y^2,z^2,u^2,zu,zx,ux,uy)$$
where $x,y$ are 2-classes, and $z,u$ are 5-classes. 
\end{proposition}

There is only one case with a regular Reeb vector field, and that is $\bfw=(3,1)$ in which case the relative Fano indices are $(1,2)$. Then the 7-manifold is $(S^2\times S^3)\star_{1,2}S^3_{(3,1)}$ can be realized as an $L(2;3,1)\approx \bbr\bbp^3$ lens space bundle over $\bbc\bbp^1\times\bbc\bbp^1$.


\subsubsection{$M=k(S^2\times S^3), N=\bbc\bbp^2$ blown-up at $k$ generic points with $k=1,\ldots,8$} \label{blowups}
Equivalently we write $N=N_k=\bbc\bbp^2\#k\overline{\bbc\bbp}^2$. All the K\"ahler structures have an extremal representative, but for $k=1,2$ they are not CSC. However, for $k=3,\ldots,8$ they are CSC, and hence, K\"ahler-Einstein. Notice that when $4\leq k\leq 8$ the complex automorphism group has dimension $0$, so the $\bfw$-Sasaki cone is the entire Sasaki cone. Moreover, if $5\leq k\leq 8$ the local moduli space has positive dimension, and we can choose any of the complex structures. By a theorem of Kobayashi and Ochiai \cite{KoOc73} we have $\cali_{N_k}=1$ for all $k=1,\ldots,8$. So $l_1=1,l_2=|\bfw|$, and by Corollary \ref{cali1cor} there are no regular Reeb vector fields in the $\bfw$-Sasaki cone with $\bfw\neq (1,1)$. In particular, if $4\leq k\leq 8$, there are no regular Reeb vector fields in the Sasaki cone. Generally, these are $L(|\bfw|;w_1,w_2)$ lens space bundles over $N_k$. Of course, the case $\bfw=(1,1)$ is just an $S^1$-bundle over $N_k\times \bbc\bbp^1$ with the product complex structure which is automatically regular. These were studied in \cite{BG00a}. Let $\cals_k$ denote the total space of the principal $S^1$-bundle over $N_k$ corresponding to the anticanonical line bundle $K^{-1}$ on $N_k$. By a well-known result of Smale $\cals_k$ is diffeomorphic to the $k$-fold connected sum $k(S^2\times S^3)$. We consider the join $\cals_k\star_{1,|\bfw|}S^3_\bfw$. The case $\bfw=(1,1)$ was studied in \cite{BG00a} where it is shown to have a Sasaki-Einstein metric when $3\leq k\leq 8$. Moreover, in this case we have determined the integral cohomology ring (see Theorem 5.4 of \cite{BG00a}). Here we generalize this result.

\begin{theorem}\label{Sk7man}
The integral cohomology ring of the 7-manifolds $M^7_{k,\bfw}=\cals_k\star_{1,|\bfw|} S^3_\bfw$ is given by
$$H^q(M^7_{k,\bfw},\bbz)\approx \begin{cases}
                    \bbz & \text{if $q=0,7$;} \\
                    \bbz^{k+1} & \text{if $q=2,5$;} \\
                    \bbz^k_{|\bfw|}\times \bbz_{w_1w_2} & \text{if $q=4;$} \\
                     0 & \text{if otherwise}, 
                     \end{cases}$$
with the ring relations determined by $\gra_i\cup \gra_j=0, w_1w_2s^2=0, |\bfw|\gra_i\cup
s=0,$ where $\gra_i,s$ are the $k+1$ two classes with $i=1,\cdots k.$
\end{theorem}

\begin{proof}
As before the $E_2$ term of the Leray-Serre spectral sequence of the top fibration of diagram (\ref{orbifibrationexactseq}) is 
$$E^{p,q}_2=H^p(\mathsf{B}S^1,H^q(\cals_k\times S^3_\bfw,\bbz))\approx \bbz[s]\otimes\prod_i\grL[\gra_i,\grb_i,\grg]/ \gI,$$
where $\gra_i,\grb_j,\grg$ have degrees $2,3,3$, respectively, and $\gI$ is the ideal generated by the relations $\gra_i\cup \grb_i=\gra_j\cup \grb_j,\gra_i\cup\gra_j=\grb_i\cup\grb_j=0$ for all $i,j$, $\gra_i\cup\grb_j=0$ for $i\neq j$ and $\grg^2=0$.

Consider the lower product fibration of diagram (\ref{orbifibrationexactseq}). As in the previous case the first non-vanishing differential of the second factor is $d_4$, and as in that case $d_4(\grg)=w_1w_2s^2_2$. For the first factor we know from Smale's classification of simply connected spin 5-manifolds that $\cals_k$ is diffeomorphic to the $k$-fold connected sum $k(S^2\times S^3)$. Moreover, since $N=\bbc\bbp^2\# k\overline{\bbc\bbp}^2$, the first factor fibration is
$$k(S^2\times S^3)\ra{1.8} \bbc\bbp^2\# k\overline{\bbc\bbp}^2\ra{1.8} \mathsf{B}S^1.$$
Here the first non-vanishing differential is $d_2(\grb_i)=\gra_i\otimes s$.
Again from the commutativity of diagram (\ref{orbifibrationexactseq}) for the top fibration we have $d_2(\grb_i)=|\bfw|\gra_i\otimes s$ at the $E_2$ level and $d_4(\grg)=w_1w_2s^2$ at the $E_4$ level. One easily sees that the $k+1$ $2$-classes $\gra_i\in E_2^{2,0}$ and $s\in E_2^{0,2}$ live to $E_\infty$ and there is no torsion in degree $2$. Moreover, there is nothing in degree $1$, and the $3$-classes $\grb_i\in E_2^{3,0}$ and $\grg\in E_4^{3,0}$ die, so there is nothing in degree $3$. However, there is torsion in degree $4$, namely $\bbz_{|\bfw|}^k\times \bbz_{w_1w_2}$. The remainder follows from Poincar\'e duality and dimensional considerations.
\end{proof}

This generalizes Theorem 5.4 of \cite{BG00a} where the case $\bfw=(1,1)$ is treated.

\begin{remark} 
Since $|\bfw|$ and $w_1w_2$ are relatively prime, $H^4(M^7_{k,\bfw},\bbz)\approx \bbz^{k-1}_{|\bfw|}\times \bbz_{w_1w_2|\bfw|}$. We can ask the question: when can $M^7_{k,\bfw}$ and $M^7_{k',\bfw'}$ have isomorphic cohomology rings? It is interesting and not difficult to see that there is only one possibility, namely $M^7_{1,(3,2)}$ and $M^7_{1,(5,1)}$ in which case $H^4\approx \bbz_{30}$. 
\end{remark}

\def\cprime{$'$} \def\cprime{$'$} \def\cprime{$'$} \def\cprime{$'$}
  \def\cprime{$'$} \def\cprime{$'$} \def\cprime{$'$} \def\cprime{$'$}
  \def\cdprime{$''$} \def\cprime{$'$} \def\cprime{$'$} \def\cprime{$'$}
  \def\cprime{$'$}
\providecommand{\bysame}{\leavevmode\hbox to3em{\hrulefill}\thinspace}
\providecommand{\MR}{\relax\ifhmode\unskip\space\fi MR }
\providecommand{\MRhref}[2]{%
  \href{http://www.ams.org/mathscinet-getitem?mr=#1}{#2}
}
\providecommand{\href}[2]{#2}

\end{document}